\documentclass[11pt,a4paper]{article}

\usepackage{fullpage}
\usepackage[utf8]{inputenc}
\usepackage[T1]{fontenc}
\usepackage{textcomp}
\usepackage[french,english]{babel}
\usepackage{amsmath}
\usepackage{amssymb}
\usepackage{enumerate}
\usepackage{amsthm}
\usepackage{graphicx}
\usepackage{epic}
\usepackage{url}
\usepackage{mathrsfs}
\usepackage{fancyhdr}
\usepackage{xcolor}
\usepackage{fancybox}
\usepackage{multienum}

\theoremstyle{definition}
\newtheorem{Def}{Definition}
\newtheorem{Rem}[Def]{Remark}

\theoremstyle{plain}
\newtheorem{Prop}[Def]{Proposition}
\newtheorem{Lemma}[Def]{Lemma}
\newtheorem{Thm}[Def]{Theorem}
\newtheorem{Cor}[Def]{Corollary}

\DeclareMathOperator{\R}{\mathbb{R}}
\DeclareMathOperator{\N}{\mathbb{N}}
\DeclareMathOperator{\Z}{\mathbb{Z}}

\DeclareMathOperator{\C}{\mathbb{C}}

\title{About the Uniform H\"{o}lder Continuity of\\ Generalized Riemann Function}
\author{F. Bastin \and S. Nicolay \and L. Simons}
\date{\today}

\begin{document}

\maketitle

\begin{abstract}
In this paper, we study the uniform H\"{o}lder continuity of the generalized Riemann function~$R_{\alpha,\beta}$ (with $\alpha>1$ and $\beta>0$) defined by
\[
R_{\alpha,\beta}(x)=\sum_{n=1}^{+\infty}\frac{\sin(\pi n^\beta x)}{n^\alpha},\quad x\in\R,
\]
using its continuous wavelet transform. In particular, we show that the exponent we find is optimal. We also analyse the behaviour of~$R_{\alpha,\beta}$ as $\beta$ tends to infinity.
\end{abstract}
\bigskip


\section{Introduction}

In the $\text{19}^\text{th}$ century, Riemann introduced the function $R$ defined by
\[
R(x)=\sum_{n=1}^{+\infty}\frac{\sin(\pi n^2 x)}{n^2},\quad x\in\R,
\]
in order to construct a continuous but nowhere differentiable function (see~\cite{DBR} for some historical informations). The regularity of this function has been extensively studied by many authors. In 1916, Hardy~\cite{H} showed that $R$ is not differentiable at irrational numbers and at some rational numbers. In the seventies, Gerver~\cite{G1} and other people~\cite{I,M,Q,S,HT} proved that $R$ is only differentiable at the rational numbers $(2p+1)/(2q+1)$ (with $p\in\Z$ and $q\in\N$) with a derivative equals to $-1/2$.
\bigskip

The H\"{o}lder spaces allow to define a notion of smoothness or regularity for a function. In some way, it is an ``intermediate level'' between continuity and differentiability. Following~\cite{Ja,JMR,Da,T}, we adopt the next definition for H\"{o}lder spaces.

\begin{Def}
Let $\alpha\in[0,1)$, $f\in L^{\infty}(\R)$ and $x_0\in\R$. 
\begin{enumerate}[(1)]\itemsep=0cm
\item The function $f$ belongs to $C^\alpha(x_0)$ if there exists $C>0$ and $\varepsilon>0$ such that
\[
|f(x)-f(x_0)|\leq C |x-x_0|^{\alpha}
\]
for all $x\in(x_0-\varepsilon,x_0+\varepsilon)$. In this case, we say that $f$ is {\em H\"{o}lder continuous with exponent~$\alpha$ at~$x_0$}.
\item The function $f$ belongs to $C^{\alpha}(\R)$ if there exists $C>0$ such that
\[
|f(x)-f(y)|\leq C |x-y|^{\alpha}
\]
for all $x,y\in\R$. In this case, we say that $f$ is {\em uniformly H\"{o}lder continuous with exponent~$\alpha$ (on~$\R$)}.
\end{enumerate}
\end{Def}

\noindent The spaces defined above are embedded: if $\alpha<\beta$ for $\alpha,\beta\in[0,1)$, then $C^{\beta}(x_0)\subset C^\alpha(x_0)$ for any $x_0\in\R$ and $C^{\beta}(\R)\subset C^\alpha(\R)$. This property allows to define a notion of regularity, known as H\"{o}lder exponent.

\begin{Def}
Let $f\in L^\infty(\R)$ and let $x_0\in\R$.
\begin{enumerate}[(1)]\itemsep=0cm
\item The {\em H\"{o}lder exponent of $f$ at $x_0$} is
\[
H_f(x_0)=\sup\left\{\alpha\in[0,1):f\in C^{\alpha}(x_0)\right\}.
\]
\item The {\em uniform H\"{o}lder exponent of $f$ (on $\R$)} is
\[
H_f(\R)=\sup\left\{\alpha\in[0,1):f\in C^{\alpha}(\R)\right\}.
\]
\end{enumerate}
\end{Def}

\noindent Following this definition, if $f$ is differentiable, then $H_f(\R)=1$. Moreover, $H_f(\R)<1$ implies that $f$ is not differentiable. However, there exist non-differentiable functions with a uniform H\"{o}lder exponent equal to~$1$; the Takagi function (see~\cite{Ta,SS}) is a famous example.

\bigskip

Based on a work with Littlewood~\cite{HL}, Hardy~\cite{H} showed that~$R$ is not H\"{o}lder continuous with exponent~$3/4$ at irrational numbers and at some rational numbers. Using the continuous wavelet transform (of~$R$), Holschneider and Tchamitchian~\cite{HT} established that $R$ is uniformly H\"{o}lder continuous with exponent $1/2$ and gave some results about its H\"{o}lder continuity at some particular points. With some similar techniques, Jaffard and Meyer~\cite{Ja,JM} determined the H\"{o}lder exponent of $R$ at each point and proved that $R$ is a multifractal function, i.e. that the function $x\mapsto H_R(x)$ is not constant.

\bigskip
A generalization of $R$ is given by the function $R_{\alpha,\beta}$ defined by
\begin{equation}\label{Rgen}
R_{\alpha,\beta}(x)=\sum_{n=1}^{+\infty} \frac{\sin(\pi n^\beta x)}{n^\alpha},\quad x\in\R,
\end{equation}
with $\alpha>1$ and $\beta>0$. Other generalizations of $R$ are possible; for example, one can replace the element~$n^\beta$ in the definition of $R_{\alpha,\beta}$ by a polynomial with entire coefficients (see~\cite{CU,Q}).

\bigskip

The function~$R_{\alpha,\beta}$ defined in~\eqref{Rgen} is clearly continuous and bounded on $\R$. If $\beta\in(0,\alpha-1)$, it is easy to check that $R_{\alpha,\beta}$ is continuously differentiable on $\R$ (because the series of derivatives uniformly converge on $\R$). If $\beta\geq\alpha+1$, Luther~\cite{L} proved that $R_{\alpha,\beta}$ is nowhere differentiable. If $\beta\in[\alpha-1,\alpha+1)$, several partial results about the differentiability of~$R_{\alpha,\beta}$ are known (see~\cite{Q,L}). Moreover, some results are also known for the cases $\beta=2$ (see~\cite{H,Ja}), $\beta=3$ (see~\cite{G}) and $\beta\in\N\setminus\{0\}$ (see~\cite{CU1}). Concerning the H\"{o}lder continuity and also the H\"{o}lder exponent of~$R_{\alpha,\beta}$, several particular cases have been studied (see~\cite{B,Ja,JM,JJ,CU1,U}).

\bigskip
In this paper, we study the uniform H\"{o}lder continuity of $R_{\alpha,\beta}$ with $\beta\geq\alpha-1$ in order to complete and generalize a result of Johnsen~\cite{JJ} in 2010 which claims that, if $\beta>\alpha-1$, then~$R_{\alpha,\beta}$ is uniformly H\"{o}lder continuous with an exponent superior or equal to~$(\alpha-1)/\beta$. To achieve this, we use some techniques different from the ones of Johnsen. Our approach is based on the continuous wavelet transform of~$R_{\alpha,\beta}$ related to the Lusin wavelet, and is similar to the ones used to obtain the H\"{o}lder continuity of~$R$ in~\cite{JMR,Ja,HT}. This method has two advantages: we can consider both the cases $\beta=\alpha-1$ and $\beta>\alpha-1$ to study the uniform H\"{o}lder continuity of~$R_{\alpha,\beta}$ and then show the optimality of the so obtained exponent. In other words, we calculate the uniform H\"{o}lder exponent of~$R_{\alpha,\beta}$ for $\beta\geq\alpha-1$. These results are summarized in the following theorem.
\begin{Thm}\label{Main}
We have
\[
H_{R_{\alpha,\beta}}(\R)=\left\{
\begin{array}{ll}\vspace{1.5ex}
1&\text{if }\beta=\alpha-1\\ 
\displaystyle\frac{\alpha-1}{\beta}&\text{if }\beta>\alpha-1
\end{array}\right..
\]
\end{Thm}
\bigskip

If we fix $\alpha>1$, the uniform H\"{o}lder exponent of~$R_{\alpha,\beta}$ decreases to~$0$ as $\beta$ increases to infinity. In order to illustrate this phenomenon, we give the graphical representation of~$R_{\alpha,\beta}$ for some~$\beta$. For $\beta$ large enough, we can observe that~$R_{\alpha,\beta}$ seems to be the function $x\mapsto \sin(\pi x)$ with some noise or fluctuations all around. In fact, this function is simply the first term of the series defining~$R_{\alpha,\beta}$. We show that $R_{\alpha,\beta}$ can be, on average, compared to the function $x\mapsto\sin(\pi x)$ and we measure the amplitude of these fluctuations.

\bigskip
The paper is organized as follows. In Section~\ref{CWT}, we recall some helpful properties about the continuous wavelet transform and the tool that it provides to study the H\"{o}lder continuity of a function. We will extensively take advantage of the properties of the Lusin wavelet. The proof of Theorem~\ref{Main} is given in Section~\ref{ProofMain}. We analyse in Section~\ref{Graphic} the behaviour of~$R_{\alpha,\beta}$ as $\beta$ increases. We present the graphical representation of~$R_{2,\beta}$ for some particular values of~$\beta$. In section~\ref{FinalRem}, we give some additional comments about the more general case of nonharmonic Fourier series. We also show the limitations of the Lusin wavelet to investigate the research of the maximal possible H\"{o}lder exponent of~$R_{\alpha,\beta}$ at a point.

\section{H\"{o}lder continuity and continuous wavelet transform}\label{CWT}

Let us recall some notions about the continuous wavelet transform and the H\"{o}lder continuity of a function (see~\cite{Da,Ja,JMR,T,Ho,HT}). The natural space associated to the continuous wavelet transform is the Hilbert space~$L^2(\R)$. Such a setting is of no interest for the function~$R_{\alpha,\beta}$, since it does not belong to $L^2(\R)$. As~$R_{\alpha,\beta}$ is a continuous and bounded function on~$\R$, the continuous wavelet transform of a function of~$L^\infty(\R)$ is more appropriate.

\begin{Def}
The function $\psi$ is a {\em wavelet} if $\psi\in L^1(\R)\cap L^2(\R)$ and $\hat{\psi}(0)=0$, where $\hat{\psi}$ denotes the Fourier transform of $\psi$:
\[
\hat{\psi}(\xi)=\int_{\R}e^{-ix\xi}\psi(x)\,dx,\quad\xi\in\R.
\]
Using the wavelet $\psi$, the {\em continuous wavelet transform} of a function $f\in L^\infty(\R)$ is the function~$\mathcal{W}_\psi f$ defined by
\[
\mathcal{W}_\psi f(a,b)=\int_{\R} f(x)\,\frac{1}{a}\,\overline{\psi}\left(\frac{x-b}{a}\right)\,dx,\quad a>0,\,b\in\R,
\]
where $\overline{\psi}$ denotes the complex conjugate of $\psi$.
\end{Def}

\bigskip
In order to study the uniform H\"{o}lder continuity of~$R_{\alpha,\beta}$, we will use a peculiar wavelet, known as the Lusin wavelet:
\begin{equation}\label{Lusin}
\psi(x)=\frac{1}{\pi(x+i)^2},\quad x\in\R.
\end{equation}
Since
\[
\hat{\psi}(\xi)=\left\{\begin{array}{ll}
-2\xi e^{-\xi}&\text{if }\xi\geq 0\\
0&\text{if }\xi<0
\end{array}\right.,
\]
this wavelet belongs to the second Hardy space
\[
H^2(\R)=\left\{f\in L^2(\R): \hat{f}=0\,\text{ a.e. on }(-\infty,0)\right\}.
\]
Such a property will be useful to obtain a simple explicit expression of~$\mathcal{W}_\psi R_{\alpha,\beta}$ (in comparison with the derivatives of a gaussian function for example).

\bigskip
An exact reconstruction formula exists in such a situation: if $\psi$ belongs to~$H^2(\R)$ and if $f$ belongs to a certain class of continuous and bounded functions on~$\R$, we can recover~$f$ from $\mathcal{W}_\psi f$ using a second wavelet satisfying some additional properties. This result is strongly inspired by Proposition 2.4.2 in~\cite{Da} and Theorem 2.2 in~\cite{HT}. For the sake of completeness, we give in the appendix a proof based on the ideas of~\cite{Da,HT,Ho} and adapted to our case.

\begin{Thm}\label{LusinRecons}
Let $\psi$ be a wavelet which belongs to $H^2(\R)$. Let $\varphi$ be a differentiable wavelet such that $x\mapsto x\varphi(x)$ is integrable on~$\R$, such that $D\varphi$ is square integrable on~$\R$ and such that \begin{equation}\label{CondAdm}
\int_0^{+\infty}\overline{\hat{\psi}}(\xi)\hat{\varphi}(\xi)\,\frac{d\xi}{\xi}=1.
\end{equation}
If $f$ is a continuous and bounded function on $\R$ and is weakly oscillating around the origin, i.e. such that
\[
\lim_{r\to +\infty}\; \sup_{x\in\R}\left|\frac{1}{2r}\int_{x-r}^{x+r} f(t)\,dt\right|=0,
\]
then we have
\[
f(x)=\lim_{\substack{\varepsilon\to 0^+\\ r\to+\infty}} 2\int_\varepsilon^r\left(\int_{-\infty}^{+\infty}\mathcal{W}_\psi f(a,b)\,\frac{1}{a}\varphi\left(\frac{x-b}{a}\right)\,db\right)\,\frac{da}{a}
\]
for all $x\in\R$.
\end{Thm}

Thanks to this reconstruction formula, the H\"{o}lder continuity of a function can be characterized with its continuous wavelet transform, provided that the wavelet satisfies some additional conditions. We will use the following result to study the H\"{o}lder continuity of the generalized Riemann function (see~\cite{JMR,Ja,HT}). 

\begin{Thm}\label{CaractUnif}
Let $\alpha\in(0,1)$, let $\psi$ be a wavelet such that $x\mapsto x^\alpha\psi(x)$ is integrable on~$\R$ and let~$f$ be a function as in Theorem~\ref{LusinRecons}.
\begin{enumerate}[(1)]\itemsep=0cm
\item We have $f\in C^{\alpha}(\R)$ if and only if there exists $C>0$ such that 
\[
|\mathcal{W}_\psi f(a,b)|\leq C\,a^{\alpha}
\]
for all $a>0$ and $b\in\R$.
\item Let $x_0\in\R$. If $f\in C^{\alpha}(x_0)$, then there exist $C>0$ and $\eta>0$ such that 
\[
|\mathcal{W}_\psi f(a,b)|\leq C\,a^{\alpha}\left(1+\left(\frac{|b-x_0|}{a}\right)^\alpha\right)
\]
for all $a\in(0,\eta)$ and $b\in(x_0-\eta,x_0+\eta)$. Conversely, if there exist $\alpha'\in(0,\alpha)$, $C>0$ and $\eta>0$ such that
\[
|\mathcal{W}_\psi f(a,b)|\leq C\,a^{\alpha}\left(1+\left(\frac{|b-x_0|}{a}\right)^{\alpha'}\right)
\]
for all $a\in(0,\eta)$ and $b\in(x_0-\eta,x_0+\eta)$, then $f\in C^\alpha(x_0)$.\hfill $\square$
\end{enumerate}
\end{Thm}
\bigskip

\begin{Rem}
Let us note that the necessary conditions in Theorem~\ref{CaractUnif} do not need all the hypotheses on the function~$f$: the continuity and the weak oscillation around the origin of~$f$ are not useful for these implications.
\end{Rem}
\bigskip

The generalized Riemann function and the Lusin wavelet satisfy the conditions of the two previous theorems. Indeed, we know that~$R_{\alpha,\beta}$ is continuous and bounded and that the Lusin wavelet~$\psi$ belongs to~$H^2(\R)$. Moreover, $R_{\alpha,\beta}$ is weakly oscillating around the origin because
\[
\left|\frac{1}{2r}\int_{x-r}^{x+r} R_{\alpha,\beta}(t)\,dt\right|
\leq\left|\frac{1}{2r}\sum_{n=1}^{+\infty}\frac{\cos((x-r)\pi n^\beta)-\cos((x+r)\pi n^\beta)}{\pi n^{\alpha+\beta}}\right|
\leq\frac{\zeta(\alpha+\beta)}{\pi r}
\]
for all $x\in\R$ and $r>0$, and $x\mapsto x^\alpha\psi(x)$ is clearly integrable for $\alpha\in(0,1)$. Besides, it is easy to find a differentiable wavelet~$\varphi$ such that $x\mapsto x\varphi(x)$ is integrable on~$\R$, such that $D\varphi$ is square integrable on~$\R$ and such that
\[
\int_0^{+\infty}\hat{\varphi}(\xi)e^{-\xi}\,d\xi=-\frac{1}{2}.
\]
In the following, $\psi$ will systematically denote the Lusin wavelet (see~\eqref{Lusin}).

\section{H\"{o}lder continuity of generalized Riemann function}\label{ProofMain}

Since we know that the function $R_{\alpha,\beta}$ is continuously differentiable on $\R$ if $\alpha>1$ and $\beta\in(0,\alpha-1)$, we may assume $\beta\geq\alpha-1$ in the study of the uniform H\"{o}lder continuity of $R_{\alpha,\beta}$. To prove Theorem~\ref{Main}, we first need to determine the continuous wavelet transform of~$R_{\alpha,\beta}$ related to the Lusin wavelet, as in~\cite{JMR,Ja,HT} where the case $\alpha=\beta=2$ is treated. 

\begin{Prop}
We have
\begin{eqnarray}\label{Wab}
\mathcal{W}_\psi R_{\alpha,\beta}(a,b)=ia\pi\sum_{n=1}^{+\infty}\frac{e^{i\pi n^\beta(b+ia)}}{n^{\alpha-\beta}}
\end{eqnarray}
for all $a>0$ and $b\in\R$.
\end{Prop}

\begin{proof}
We can write
\[
R_{\alpha,\beta}(x)=\frac{1}{2}\left(T_{\alpha,\beta}(x)-\widetilde{T}_{\alpha,\beta}(x)\right)
\]
for $x\in\R$ with
\[
T_{\alpha,\beta}(x)=-i\sum_{n=1}^{+\infty}\frac{e^{i\pi n^{\beta}x}}{n^\alpha}\qquad\text{and}\qquad\widetilde{T}_{\alpha,\beta}(x)=T_{\alpha,\beta}(-x).
\]
In other words, $R_{\alpha,\beta}$ is the odd part of $T_{\alpha,\beta}$. 

Let us fix $a>0$ and $b\in\R$. We have
\[
\mathcal{W}_\psi T_{\alpha,\beta}(a,b)
=\int_{\R}T_{\alpha,\beta}(x)\,\frac{1}{a}\overline{\psi}\left(\frac{x-b}{a}\right)\,dx
=\frac{a}{\pi}\int_{\R}\frac{T_{\alpha,\beta}(x)}{(x-(b+ia))^2}\,dx.
\]
For $\eta>0$ and $r>0$, let us denote by $\gamma_{\eta,r}$ the closed path formed by the juxtaposition of the two following ones: the first path describes the segment $[-r+i\eta,r+i\eta]$ and the second one the half-circle of center~$i\eta$ and radius $r$ included in $H=\{z\in\C:\Im z>0\}$. The function $T_{\alpha,\beta}$ is holomorphic on $H$ because the series uniformly converges on every compact set of $H$. As the point $b+ia$ is situated inside the curve described by $\gamma_{\eta,r}$ for $\eta\in(0,a)$ and $r>a$, we obtain
\begin{eqnarray*}
\mathcal{W}_\psi T_{\alpha,\beta}(a,b)
&=&\frac{a}{\pi}\lim_{r\to+\infty}\lim_{\eta\to 0^+}\int_{\gamma_{\eta,r}}\frac{T_{\alpha,\beta}(z)}{(z-(b+ia))^2}\,dz\\
&=&2ia\,(DT_{\alpha,\beta})(b+ia)\\
&=&2ia\pi\sum_{n=1}^{+\infty}\frac{e^{i\pi n^\beta(b+ia)}}{n^{\alpha-\beta}},
\end{eqnarray*}
thanks to Cauchy's integral formula. Similarly, the continuous wavelet transform of $\widetilde{T}_{\alpha,\beta}$ is given by
\[
\mathcal{W}_\psi\widetilde{T}_{\alpha,\beta}(a,b)
=\int_{\R}T_{\alpha,\beta}(-x)\,\frac{1}{a}\overline{\psi}\left(\frac{x-b}{a}\right)\,dx
=\frac{a}{\pi}\lim_{r\to+\infty}\lim_{\eta\to 0^+}\int_{\gamma_{\eta,r}}\frac{T_{\alpha,\beta}(z)}{(z-(-b-ia))^2}\,dz=0
\]
by homotopy invariance, because the point $-b-ia$ does not belong to $H$. We thus have the conclusion.
\end{proof}

Let us now analyse $\mathcal{W}_\psi R_{\alpha,\beta}$ in order to study the uniform H\"{o}lder continuity of $R_{\alpha,\beta}$ with Theorem~\ref{CaractUnif}. We have
\begin{equation}\label{WabMaj}
|\mathcal{W}_\psi R_{\alpha,\beta}(a,b)|\leq a\pi\sum_{n=1}^{+\infty}\frac{e^{-a\pi n^\beta}}{n^{\alpha-\beta}}
=|\mathcal{W}_\psi R_{\alpha,\beta}(a,0)|
\end{equation}
for $a>0$ and $b\in\R$. The function $f_{\alpha,\beta}:x\mapsto x^{\beta-\alpha}\,e^{-a\pi x^\beta}$ is differentiable on $(0,+\infty)$ and 
\[
Df_{\alpha,\beta}(x)=e^{-a\pi x^\beta}\,x^{\beta-\alpha-1}\,\left((\beta-\alpha)-a\pi\beta x^\beta\right),\quad x>0.
\]
Then, $f_{\alpha,\beta}$ is decreasing on $(0,+\infty)$ if $\beta\in[\alpha-1,\alpha)$ and on $(((\beta-\alpha)/a\pi\beta)^{1/\beta},+\infty)$ if $\beta\geq\alpha$. The next developments are mainly based on the classical comparison principle between series and integral (when the general term is decreasing), which we recall in the following lemma.
\begin{Lemma}
Let $N\in\N$ and let $f$ be a decreasing and positive function defined on $[N,+\infty)$. The series $\sum_{n=N+1}^{+\infty}f(n)$ converges if and only if $f$ is integrable on $[N,+\infty)$; in this case we have
\[
\int_{N+1}^{+\infty}f(x)\,dx\leq\sum_{n=N+1}^{+\infty}f(n)\leq\int_{N}^{+\infty}f(x)\,dx.
\]
\hfill $\square$
\end{Lemma}

We note that $f_{\alpha,\beta}$ is integrable on $(0,+\infty)$ only if $\beta>\alpha-1$. We therefore split the study of the uniform H\"{o}lder continuity and the calculus of the uniform H\"{o}lder exponent of~$R_{\alpha,\beta}$ into two cases: $\beta>\alpha-1$ and $\beta=\alpha-1$.

\begin{Prop}\label{ExpH}
If $\beta>\alpha-1$, then 
\[
H_{R_{\alpha,\beta}}(\R)=\frac{\alpha-1}{\beta}.
\]
\end{Prop}

\begin{proof}
1. Let us first consider the case $\beta\in(\alpha-1,\alpha)$. The function $f_{\alpha,\beta}$ is decreasing on $[1,+\infty)$ and we have
\[
|\mathcal{W}_\psi R_{\alpha,\beta}(a,b)|
\leq a\pi\left(e^{-a\pi}+\sum_{n=2}^{+\infty}\frac{e^{-a\pi n^\beta}}{n^{\alpha-\beta}}\right)
\leq a\pi\left(e^{-a\pi}+\int_1^{+\infty}\frac{e^{-a\pi x^\beta}}{x^{\alpha-\beta}}\,dx\right)
\]
for $a>0$ and $b\in\R$. For the second term of the right hand side of the last inequality, we obtain
\begin{equation}\label{GammaMaj}
\int_1^{+\infty}\frac{e^{-a\pi x^\beta}}{x^{\alpha-\beta}}\,dx
\leq \int_0^{+\infty}\frac{e^{-a\pi x^\beta}}{x^{\alpha-\beta}}\,dx
=\frac{1}{\beta}\pi^{\frac{\alpha-1}{\beta}}\,\Gamma\left(\frac{1+\beta-\alpha}{\beta}\right)\,a^{\frac{\alpha-1}{\beta}-1}
\end{equation}
for $a>0$, where $\Gamma$ is defined by
\[
\Gamma(x)=\int_0^{+\infty}e^{-t}\,t^{x-1}\,dt,\quad x>0,
\]
as usual. For the first term, we note that the function $a\mapsto e^{-a\pi} a^{1-\frac{\alpha-1}{\beta}}$ is bounded on $(0,+\infty)$ because $\alpha-1<\beta$. Then, there exists $C_{\alpha,\beta}>0$ such that
\[
|\mathcal{W}_\psi R_{\alpha,\beta}(a,b)|\leq C_{\alpha,\beta}\,a^{\frac{\alpha-1}{\beta}}
\]
for all $a>0$ and $b\in\R$, which implies $R_{\alpha,\beta}\in C^{\frac{\alpha-1}{\beta}}(\R)$ using Theorem~\ref{CaractUnif}. 

Let us show the optimality of this exponent $(\alpha-1)/\beta$ related to the uniform H\"{o}lder continuity. Let $C>0$ and $\eta>0$; we have
\[
|\mathcal{W}_\psi R_{\alpha,\beta}(a,0)|=a\pi\sum_{n=1}^{+\infty}\frac{e^{-\pi n^\beta a}}{n^{\alpha-\beta}}
\geq a\pi\int_1^{+\infty}\frac{e^{-a\pi x^\beta}}{x^{\alpha-\beta}}\,dx
=\frac{1}{\beta}\,(a\pi)^{\frac{\alpha-1}{\beta}}\,\Gamma\left(\frac{\beta-\alpha+1}{\beta},a\pi\right)
\]
for $a>0$, where $\Gamma$ is the incomplete Gamma function defined by
\[
\Gamma(x,y)=\int_y^{+\infty}e^{-t}t^{x-1}\,dt,\quad (x,y)\in(0,+\infty)\times [0,+\infty).
\]
Let us recall that $\Gamma(x,0)=\Gamma(x)$ and $\Gamma(x,y)$ converges to $\Gamma(x)$ as $y\to 0^+$ for all $x>0$.
Since $\Gamma((\beta-\alpha+1)/\beta,a\pi)\to\Gamma((\beta-\alpha+1)/\beta)$ and $a^\eta\to 0$ as $a\to 0^+$, there exists $A>0$ such that, for all $a\in(0,A)$, we have
\[
|\mathcal{W}_\psi R_{\alpha,\beta}(a,0)|
>C\,a^{\frac{\alpha-1}{\beta}+\eta}.
\]
Hence the conclusion using Theorem~\ref{CaractUnif}.

\smallskip
2. Let us now consider the case $\beta\geq\alpha$ and let us write $N_a=\lfloor((\beta-\alpha)/a\pi\beta)^{1/\beta}\rfloor+1$, where~$\lfloor x\rfloor$ denotes the largest integer smaller than or equal to the real~$x$. If $a>1$, then $N_a=1$ and we can proceed as in the previous case. Let us therefore suppose that $a\in(0,1]$. We have
\begin{eqnarray*}
|\mathcal{W}_\psi R_{\alpha,\beta}(a,b)|&\leq & a\pi\left(\sum_{n=1}^{N_a}\frac{e^{-a\pi n^\beta}}{n^{\alpha-\beta}}+\sum_{n=N_a+1}^{+\infty}\frac{e^{-a\pi n^\beta}}{n^{\alpha-\beta}}\right)\\
&\leq & a\pi\left(N_a\,N_a^{\beta-\alpha}+\int_{N_a}^{+\infty}\frac{e^{-a\pi x^\beta}}{x^{\alpha-\beta}}\,dx\right)\\
&\leq & a\pi\left(\left(\left(\frac{\beta-\alpha}{\pi\beta}\right)^{\frac{1}{\beta}}+a^{\frac{1}{\beta}}\right)^{\beta-\alpha+1}a^{\frac{\alpha-1}{\beta}-1}+\int_{0}^{+\infty}\frac{e^{-a\pi x^\beta}}{x^{\alpha-\beta}}\,dx\right)\\
& \leq & a^{\frac{\alpha-1}{\beta}}\pi\left(\left(\left(\frac{\beta-\alpha}{\pi\beta}\right)^{\frac{1}{\beta}}+1\right)^{\beta-\alpha+1}+\frac{1}{\beta}\pi^{\frac{\alpha-1}{\beta}}\,\Gamma\left(\frac{1+\beta-\alpha}{\beta}\right)\right),
\end{eqnarray*}
where we have used relation~\eqref{GammaMaj} to obtain the last inequality. We then have $R_{\alpha,\beta}\in C^{\frac{\alpha-1}{\beta}}(\R)$ using Theorem~\ref{CaractUnif}.

Let us show the optimality of the exponent related to the uniform H\"{o}lder continuity. Let $C>0$ and $\eta>0$; we have 
\begin{eqnarray*}
\sum_{n=1}^{+\infty}\frac{e^{-\pi n^\beta a}}{n^{\alpha-\beta}}&\geq &
\sum_{n=N_a}^{+\infty}\frac{e^{-\pi n^\beta a}}{n^{\alpha-\beta}}\\
&\geq &
\int_{N_a}^{+\infty}\frac{e^{-a\pi x^\beta}}{x^{\alpha-\beta}}\,dx\\
&=&\frac{1}{\beta}\,(a\pi)^{\frac{\alpha-1}{\beta}-1}\int_{a\pi N_a^\beta}^{+\infty}e^{-u}\,u^{\frac{\beta-\alpha+1}{\beta}-1}\,du\\
&\geq & \frac{1}{\beta}\,(a\pi)^{\frac{\alpha-1}{\beta}-1}\,\Gamma\left(\frac{\beta-\alpha+1}{\beta},\left(\left(\frac{\beta-\alpha}{\beta}\right)^{1/\beta}+(a\pi)^{1/\beta}\right)^\beta\right)
\end{eqnarray*}
for $a>0$. As in the case $\beta\in(\alpha-1,\alpha)$, there exists $A>0$ such that, for all $a\in(0,A)$, we have
\[
|\mathcal{W}_\psi R_{\alpha,\beta}(a,0)|
>C\,a^{\frac{\alpha-1}{\beta}+\eta},
\]
hence the conclusion using once again Theorem~\ref{CaractUnif}.
\end{proof}

\begin{Rem}
In fact, by taking $b=2k$ with $k\in\Z$, we can show that $R_{\alpha,\beta}\in C^{\frac{\alpha-1}{\beta}}(2k)$ and that the exponent cannot be improved because $\mathcal{W}_\psi R_{\alpha,\beta}(a,2k)=\mathcal{W}_\psi R_{\alpha,\beta}(a,0)$ for all $a>0$. In other words, we have
\[
H_{R_{\alpha,\beta}}(2k)=\frac{\alpha-1}{\beta}.
\]
Since this quantity is strictly smaller than $1$, $R_{\alpha,\beta}$ is consequently not differentiable at~$2k$.
\end{Rem}

\begin{Prop}
We have $H_{R_{\alpha,\alpha-1}}(\R)=1$.
\end{Prop}

\begin{proof}
We have
\[
|\mathcal{W}_\psi R_{\alpha,\alpha-1}(a,b)|\leq a\pi\left(e^{-a\pi}+\int_1^{+\infty}\frac{e^{-a\pi x^{\alpha-1}}}{x}\,dx\right)
=a\pi\left(e^{-a\pi}+\frac{1}{\alpha-1}E_1(a\pi)\right)
\]
for $a>0$ and $b\in\R$, where $E_1$ is the exponential integral defined by 
\[
E_1(x)=\int_1^{+\infty}\frac{e^{-xt}}{t}\,dt,\quad x>0.
\]
Since we have
\begin{equation}\label{E1}
\frac{1}{2}\,e^{-x}\,\ln\left(1+\frac{2}{x}\right)<E_1(x)<e^{-x}\ln\left(1+\frac{1}{x}\right)
\end{equation}
for all $x>0$ (see~\cite{AS} p. 229), we obtain
\[
|\mathcal{W}_\psi R_{\alpha,\alpha-1}(a,b)|\leq a\pi\,e^{-a\pi}\left(1+\frac{1}{\alpha-1}\,\ln\left(1+\frac{1}{a\pi}\right)\right)
\]
for $a>0$ and $b\in\R$. Let us fix $\delta\in(0,1)$. There exists $A>0$ such that, for all $a\in(0,A)$, we have
\[
\frac{1}{\alpha-1}\,\frac{\ln\left(1+\frac{1}{a\pi}\right)}{\left(1+\frac{1}{a\pi}\right)^{\delta}}<1
\]
and then
\[
|\mathcal{W}_\psi R_{\alpha,\alpha-1}(a,b)|
\leq a\pi\,e^{-a\pi}\left(1+\left(1+\frac{1}{a\pi}\right)^\delta\right)
\leq a\pi\left(1+2^\delta \left(1+\left(\frac{1}{a\pi}\right)^\delta \right)\right).
\]
There also exists $A'\in(0,A)$ such that, for all $a\in(0,A')$, we have
\[
|\mathcal{W}_\psi R_{\alpha,\alpha-1}(a,b)|\leq C_{\delta}' a^{1-\delta},
\]
where $C_\delta '$ is a positive constant (depending only on $\delta$). Since the function 
\[
a\mapsto a^\delta e^{-a\pi}\left(1+\frac{1}{\alpha-1}\,\ln\left(1+\frac{1}{a\pi}\right)\right)
\]
is bounded on $[A',+\infty)$, we also have
\[
|\mathcal{W}_\psi R_{\alpha,\alpha-1}(a,b)|\leq C_{\delta}''a^{1-\delta}
\]
for $a\in[A',+\infty)$, where $C_\delta''$ is a positive constant. We thus obtain
\[
|\mathcal{W}_\psi R_{\alpha,\alpha-1}(a,b)|\leq C_\delta\,a^{1-\delta}
\]
for all $a>0$ and $b\in\R$ where $C_\delta=\max\{C_\delta',C_\delta''\}$, which implies $R_{\alpha,\alpha-1}\in C^{1-\delta}(\R)$ using Theorem~\ref{CaractUnif}.

Let us now show that this exponent of uniform H\"{o}lder continuity is optimal. Let $C>0$; we have
\[
|\mathcal{W}_\psi R_{\alpha,\alpha-1}(a,0)|\geq a\pi\int_1^{+\infty}\frac{e^{-a\pi x^{\alpha-1}}}{x}\,dx=\frac{a\pi}{\alpha-1} E_1(a\pi)\geq a\,\frac{\pi}{2(\alpha-1)}\,e^{-a\pi}\ln\left(1+\frac{2}{a\pi}\right)
\]
for all $a>0$ thanks to~\eqref{E1} and so, there exists $A>0$ such that, for all $a\in(0,A)$, we have
\[
|\mathcal{W}_\psi R_{\alpha,\alpha-1}(a,0)|>Ca,
\]
hence the conclusion using one last time Theorem~\ref{CaractUnif}.

\end{proof}

\section{Behaviour of~$R_{\alpha,\beta}$ as~$\beta$ increases}\label{Graphic}

If we fix $\alpha>1$, we know that the uniform H\"{o}lder exponent of $R_{\alpha,\beta}$ decreases as~$\beta$ increases, thanks to Theorem~\ref{Main}. Moreover, we know that this exponent is exactly the H\"{o}lder exponent of~$R_{\alpha,\beta}$ at the origin. This phenomenon is clearly illustrated in Figure~\ref{Graphique} in the case $\alpha=2$.
\bigskip

\begin{figure}[!ht]
\centering
\includegraphics[width=7.5cm]{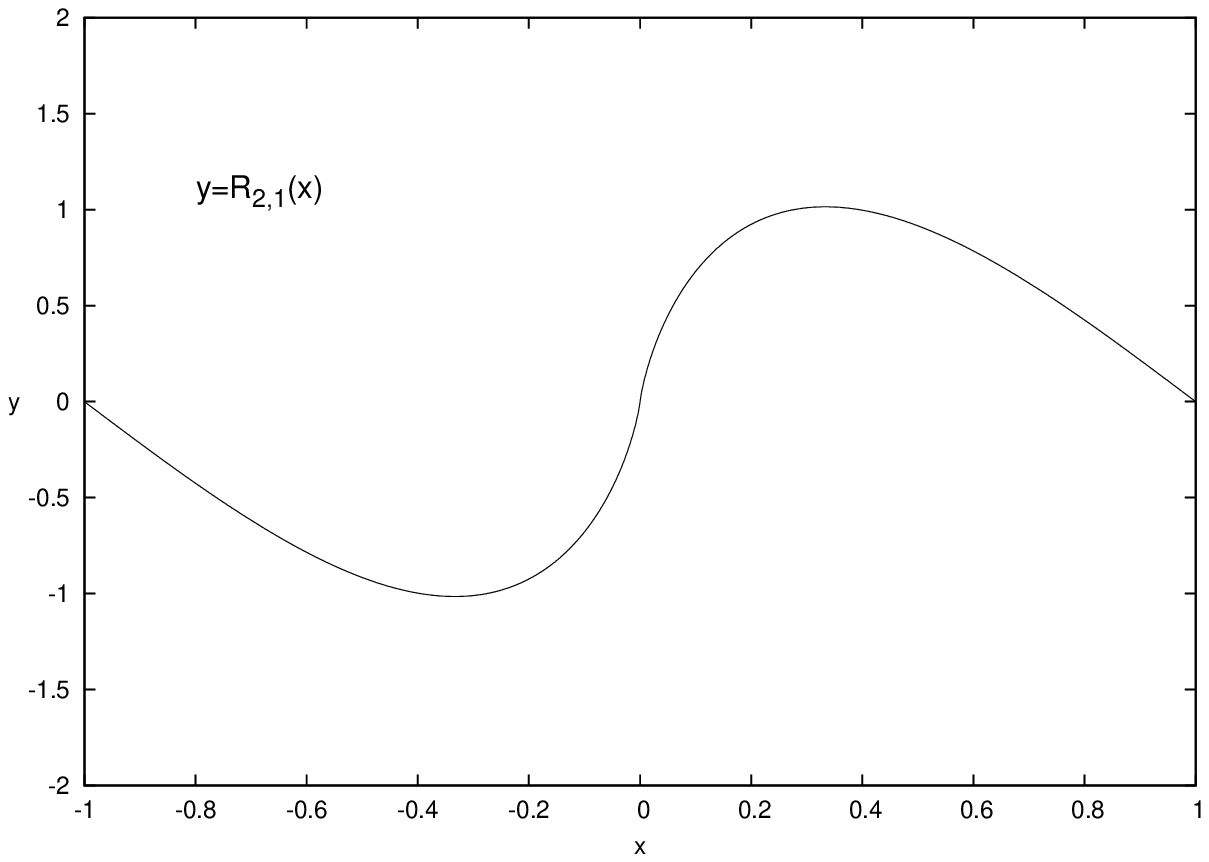}
\includegraphics[width=7.5cm]{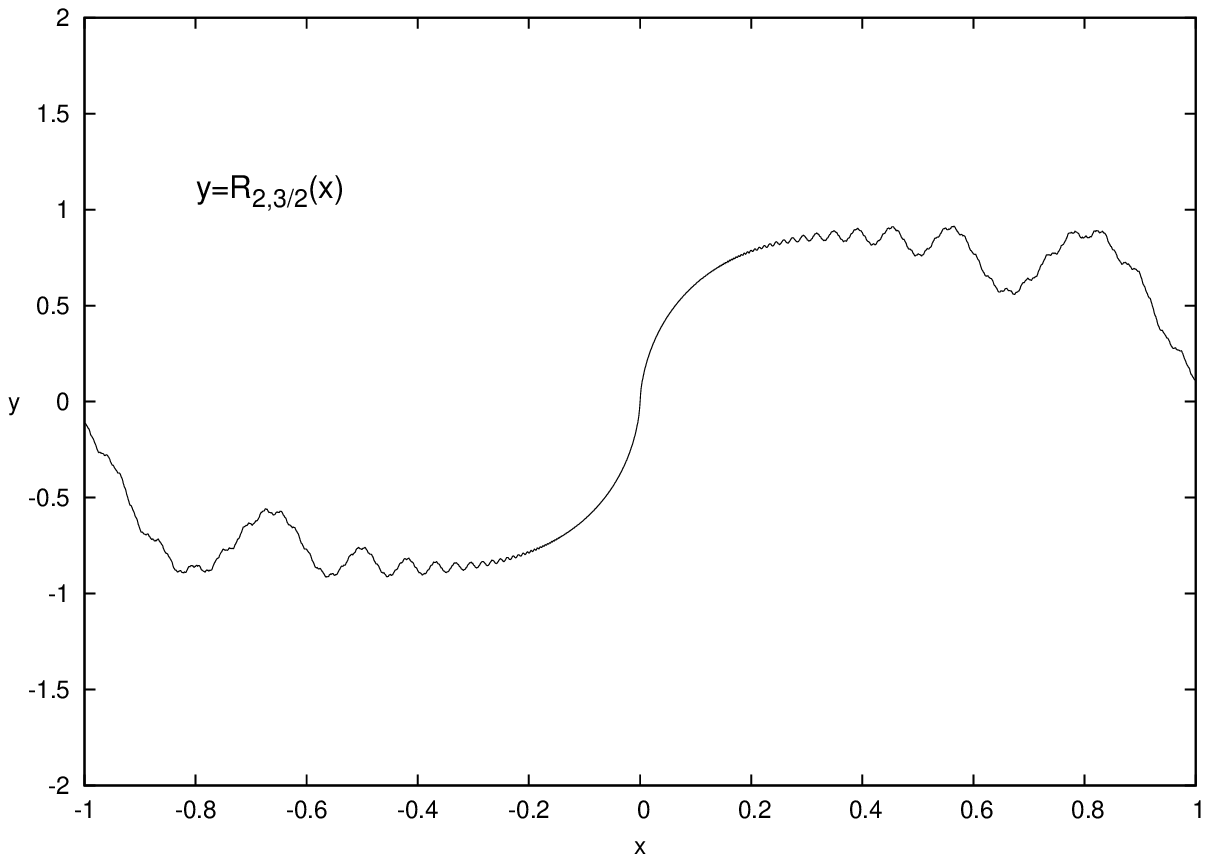}\\
\includegraphics[width=7.5cm]{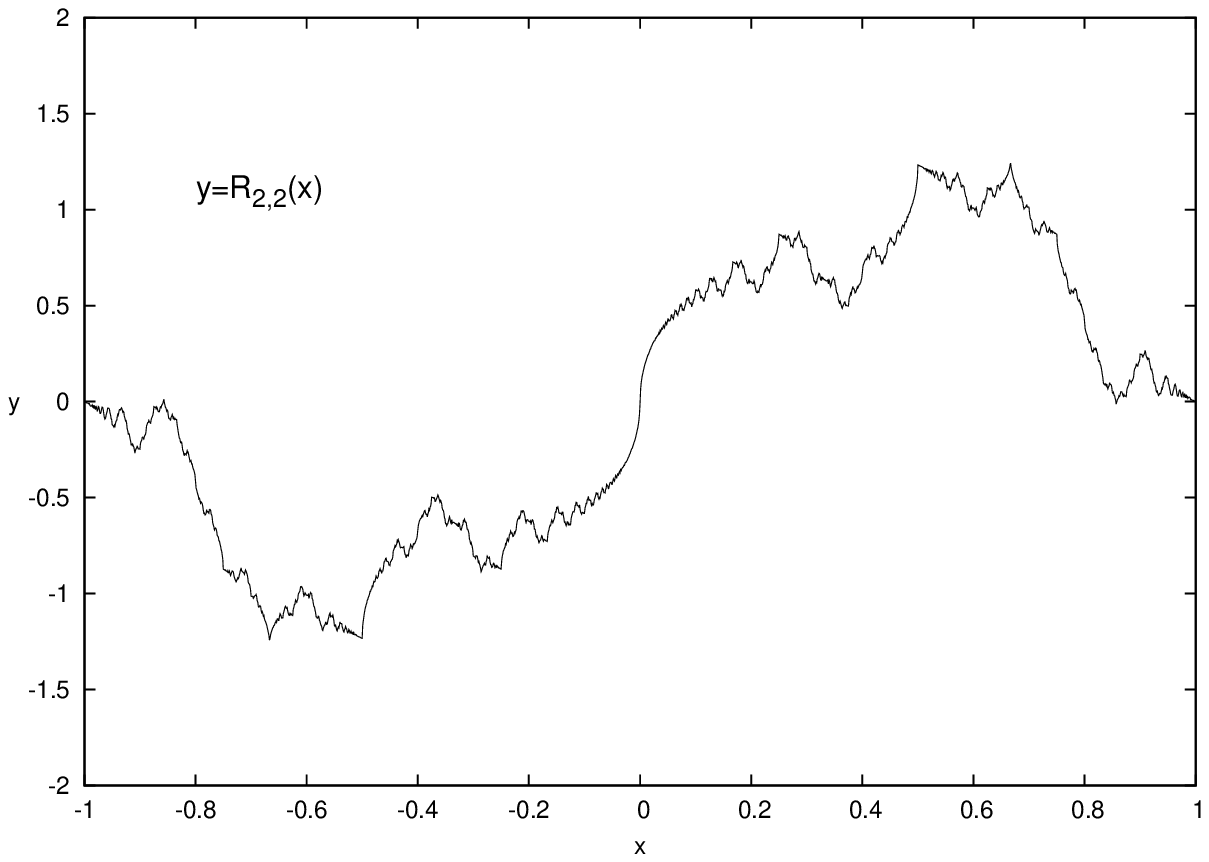}
\includegraphics[width=7.5cm]{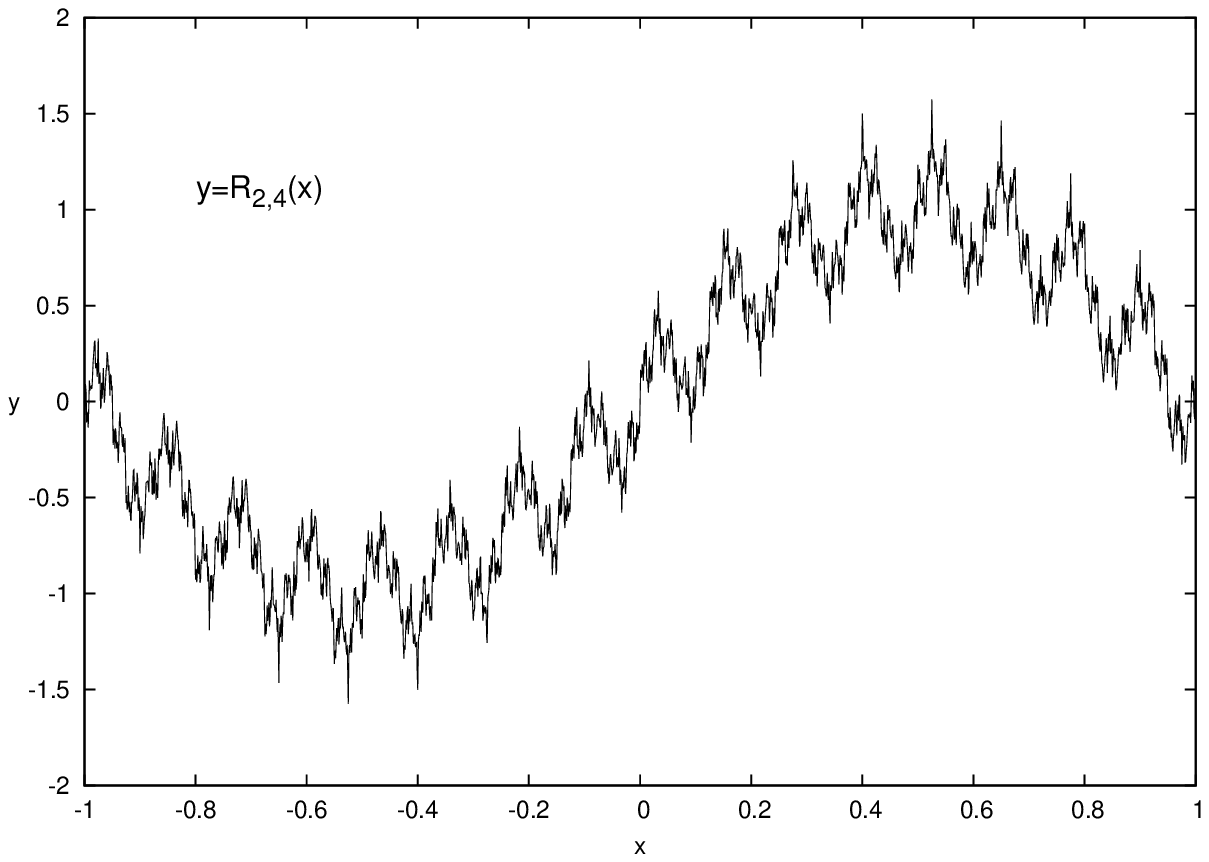}\\
\includegraphics[width=7.5cm]{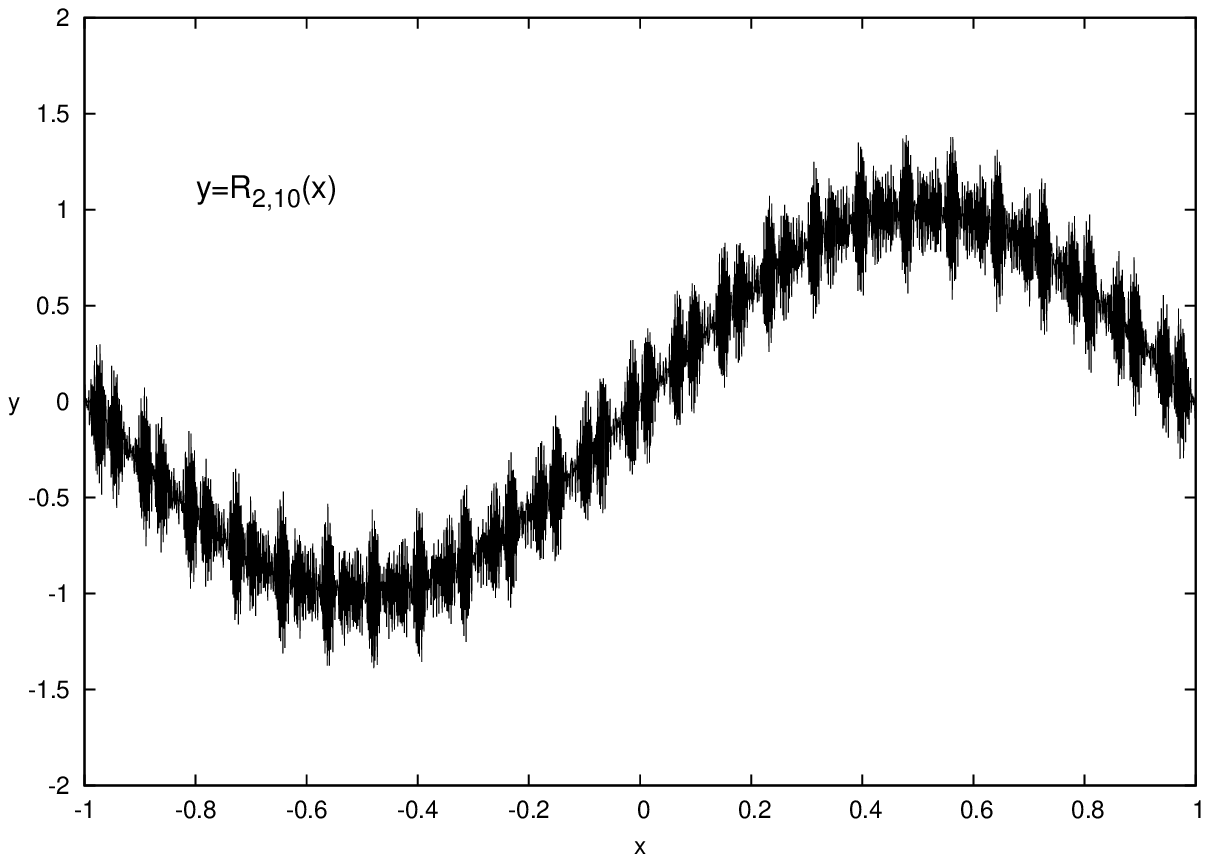}
\caption{\label{Graphique} Graphical representation of $R_{2,1}$, $R_{2,3/2}$, $R_{2,2}$, $R_{2,4}$ and $R_{2,10}$}
\end{figure}

As $\beta$ tends to infinity, we note that the graphical representation of~$R_{\alpha,\beta}$ looks like to the one of the function $s:x\mapsto\sin(\pi x)$ (in a certain sense to establish), with some noise or fluctuations all around. In the next two propositions, we give a convergence result and show that the fluctuations have a constant amplitude (i.e. independent from~$\beta$). To do so, let us recall the usual definition of the mean of an integrable function over a bounded interval.

\begin{Def}
Let $a,b\in\R$ be such that $a<b$ and let $f$ be an integrable function on~$(a,b)$. The {\em mean of the function~$f$ over the inverval~$(a,b)$} is defined by
\[
m_f^{a,b}=\frac{1}{b-a}\int_a^b f(x)\,dx.
\]
\end{Def}

\begin{Prop}\label{Mean}
Let $\alpha>1$. For all $a,b\in\R$ such that $a<b$, we have
\[
\lim_{\beta\to+\infty}m_{R_{\alpha,\beta}}^{a,b}=m_s^{a,b}.
\]
\end{Prop}

\begin{proof}
We have
\[
\left|\int_a^b(R_{\alpha,\beta}(x)-\sin(\pi x))\,dx\right|
=\left|\sum_{n=2}^{+\infty}\frac{\cos(\pi n^\beta a)-\cos(\pi n^\beta b)}{\pi n^{\alpha+\beta}}\right|\leq\frac{2}{\pi}(\zeta(\alpha+\beta)-1)
\]
and we know that $\zeta(x)\to 1$ as $x\to+\infty$, hence the conclusion.
\end{proof}

\begin{Prop}\label{Fluctuations}
Let $\alpha>1$ and let $\beta\in\N\setminus\{0\}$. The function $R_{\alpha,\beta}$ is periodic of period~$2$ and we have
\[
\int_{-1}^{1}\left(R_{\alpha,\beta}(x)-\sin(\pi x)\right)^2\,dx=\zeta(2\alpha)-1.
\]
\end{Prop}

\begin{proof}
The periodicity of~$R_{\alpha,\beta}$ is easy to check. Let us calculate the integral. By developing $x\mapsto R_{\alpha,\beta}(x)-\sin(\pi x)$ in Fourier series, we have
\[
R_{\alpha,\beta}(x)-\sin(\pi x)=\frac{a_0}{2}+\sum_{m=1}^{+\infty}\left(a_m\cos(\pi m x)+b_m\sin(\pi m x)\right)
\]
in $L^2([-1,1])$ where $a_0=a_m=0$ and
\begin{eqnarray*}
b_m&=&2\int_0^1(R_{\alpha,\beta}(x)-\sin(\pi x))\,dx\\
&=&\sum_{n=2}^{+\infty}\frac{1}{n^\alpha}\int_0^1\left(\cos(x\pi(n^\beta-m))-\cos(x\pi(n^\beta+m))\right)\,dx\\
&=&\left\{
\begin{array}{ll}\vspace{1.5ex}
\displaystyle\frac{1}{m^{\alpha/\beta}}&\text{if $m=k^\beta$ for one $k\in\N\setminus\{0,1\}$}\\
0&\text{otherwise}
\end{array}
\right.
\end{eqnarray*}
for all $m\in\N\setminus\{0\}$. Consequently, by Parseval formula, we obtain
\[
\int_{-1}^{1}\left(R_{\alpha,\beta}(x)-\sin(\pi x)\right)^2\,dx
=\sum_{m=1}^{+\infty} b_m^2=\sum_{k=2}^{+\infty}\frac{1}{k^{2\alpha}}=\zeta(2\alpha)-1.
\]
\end{proof}

The two previous propositions are illustrated in Figure~\ref{GraphiqueBeta}. Let us end this section with a simple remark about the behaviour of~$R_{\alpha,\beta}$ as $\alpha$ tends to infinity.

\begin{figure}[!ht]
\centering
\includegraphics[width=7.5cm]{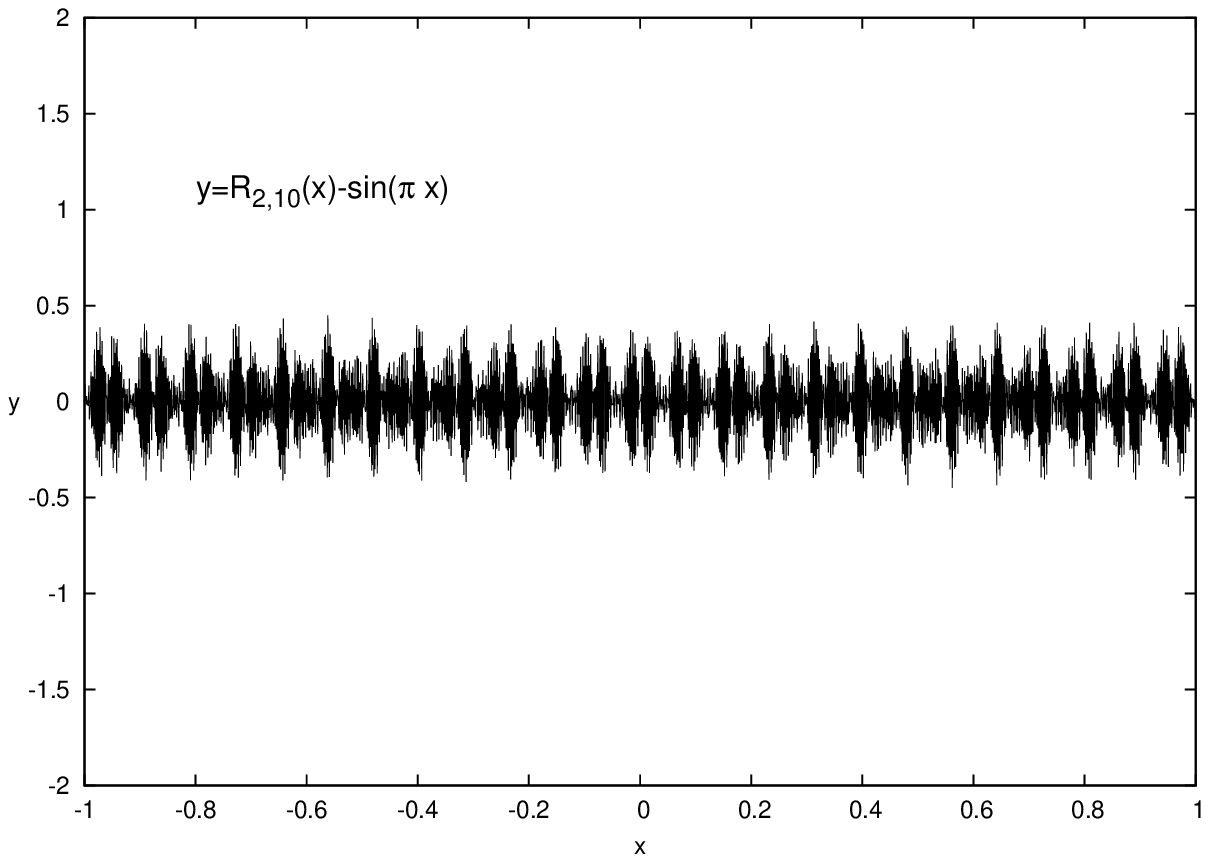}
\includegraphics[width=7.5cm]{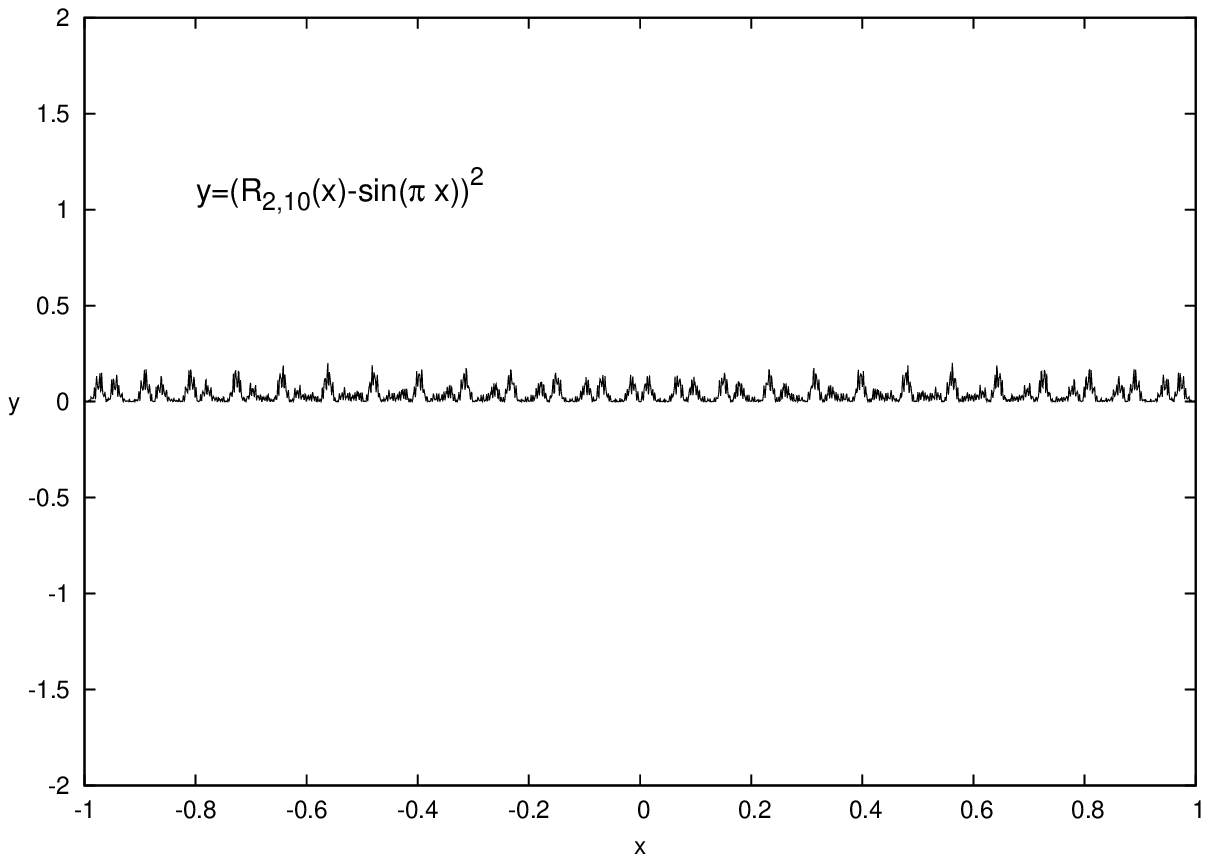}\\
\caption{\label{GraphiqueBeta} Mean value and amplitude of fluctuations of $x\mapsto R_{2,10}(x)-\sin(\pi x)$}
\end{figure}

\begin{Rem}
Proposition~\ref{Mean} is also ``satisfied'' for~$\alpha$: we have
\[
\lim_{\alpha\to +\infty}m_{R_{\alpha,\beta}}^{a,b}=m_s^{a,b}
\]
for all $\beta>0$ and all $a,b\in\R$ such that $a<b$. Moreover, by Proposition~\ref{Fluctuations}, we have
\[
\lim_{\alpha\to+\infty}\int_{-1}^1\left(R_{\alpha,\beta}(x)-\sin(\pi x)\right)^2\,dx=0
\]
for all~$\beta\in\N\setminus\{0\}$. In fact, a stronger result holds: for any fixed $\beta>0$, $R_{\alpha,\beta}$ uniformly converges on~$\R$ to~$s$ as $\alpha$ tends to infinity because we have
\[
\left|R_{\alpha,\beta}(x)-\sin(\pi x)\right|\leq\sum_{n=2}^{+\infty}\frac{1}{n^\alpha}=\zeta(\alpha)-1
\]
for all $x\in\R$. 
\end{Rem}

\section{Final remarks}\label{FinalRem}

\subsection{About nonharmonic Fourier series}

A part of Theorem~\ref{Main} can be adapted for particular nonharmonic Fourier series. Let us first recall the notion of nonharmonic Fourier series (see~\cite{L,Y,Ja1}). 

\begin{Def}
Let $\boldsymbol{a}=(a_n)_{n\in\N\setminus\{0\}}$ be a sequence of complex numbers and let $\boldsymbol{\lambda}=(\lambda_n)_{n\in\N\setminus\{0\}}$ be an increasing sequence of positive numbers which converges to infinity. A {\em nonharmonic Fourier series} (related to the sequences $\boldsymbol{a}$ and $\boldsymbol{\lambda}$) is a function $S$ defined by 
\[
S(x)=\sum_{n=1}^{+\infty}a_n\,e^{i\lambda_n x},\quad x\in\R,
\]
if the series converges.
\end{Def}

\noindent If the series $\sum_{n=1}^{+\infty}a_n$ is absolutely convergent, then the above series (related to~$S$) uniformly converges on~$\R$. We will assume that this is the case in the remainder of this discussion. Such a function $S$ is then continuous and bounded on $\R$. As for~$R_{\alpha,\beta}$, we can calculate the continuous wavelet transform of $S$ (related to the Lusin wavelet).

\bigskip
Since $\lambda_n>0$ for all $n\in\N\setminus\{0\}$, $S$ is a holomorphic function on~$H$ and we have
\[
\mathcal{W}_\psi S(a,b)=-2a\sum_{n=1}^{+\infty} a_n \lambda_n\,e^{i\lambda_n (b+ia)}
\]
for $a>0$ and $b\in\R$, similarly to~\eqref{Wab}. If we assume that there exist positive constants $C_1$, $C_2$ and $C_3$, $\alpha>1$ and $\beta>0$ such that
\[
|a_n|\leq \frac{C_1}{n^\alpha}\quad\text{and}\quad C_2 n^\beta\leq \lambda_n\leq C_3 n^\beta
\]
for all $n\in\N\setminus\{0\}$, we then obtain
\[
|\mathcal{W}_\psi S(a,b)|\leq 2a C_1 C_3\sum_{n=1}^{+\infty} \frac{e^{-C_2 a n^\beta}}{n^{\alpha-\beta}}
\]
for $a>0$ and $b\in\R$ and we recover an expression similar to the one obtained for~$|\mathcal{W}_\psi R_{\alpha,\beta}(a,b)|$ in~\eqref{WabMaj}. Using the same reasoning as in the study of the uniform H\"{o}lder continuity of~$R_{\alpha,\beta}$ with $\alpha>1$ and $\beta\geq\alpha-1$, we can formulate the following result.

\begin{Cor}
With the previous assumptions on~$\boldsymbol{a}$ and~$\boldsymbol{\lambda}$, we have $S\in C^{\frac{\alpha-1}{\beta}}(\R)$ if $\beta>\alpha-1$ and $S\in C^{1-\delta}(\R)$ for all $\delta\in(0,1)$ if $\beta=\alpha-1$.
\end{Cor}

\subsection{About the Lusin wavelet}

If $\alpha=\beta=2$, we know that the largest H\"{o}lder exponent of $R=R_{2,2}$ at a point is $3/2$ and that it is attained at the rational numbers $(2p+1)/(2q+1)$ with $p\in\Z$ and $q\in\N$ (see~\cite{JM}). The continuous wavelet transform related to the Lusin wavelet of $R$ does not allow to find this exponent.
\smallskip

Indeed, for $a>0$, we have
\[
\mathcal{W}_\psi R(a,1)=ia\pi\sum_{n=1}^{+\infty}e^{i\pi n^2(1+ia)}=ia\pi\sum_{n=1}^{+\infty}(-1)^n e^{-a\pi n^2}
=\frac{ia\pi}{2}\left(\sum_{n\in\Z} e^{i\pi n}\, e^{-a\pi n^2} -1\right)
\]
and, by the Poisson summation formula,
\[
|\mathcal{W}_\psi R(a,1)|=\frac{a\pi}{2}\left|\sum_{n\in\Z} \frac{1}{\sqrt{a}}e^{-\frac{(\pi+n)^2}{4a\pi}}-1\right|=\frac{a\pi}{2}\left|\frac{e^{-\frac{\pi}{4a}}}{\sqrt{a}}\left(1+2\sum_{n=1}^{+\infty} e^{-\frac{n^2}{4a\pi}}\cosh\left(\frac{n}{2a}\right)\right)-1\right|.
\]
Let $C>0$ and $\eta>0$. We have
\[
\lim_{a\to 0^+}\frac{e^{-\frac{\pi}{4a}}}{\sqrt{a}}\left(1+2\sum_{n=1}^{+\infty} e^{-\frac{n^2}{4a\pi}}\cosh\left(\frac{n}{2a}\right)\right)=0
\] 
since we have
\[
2\sum_{n=8}^{+\infty} e^{-\frac{n^2}{4a\pi}}\cosh\left(\frac{n}{2a}\right)
\leq\int_{7}^{+\infty} e^{-\frac{x^2}{4a\pi}}\left(1+e^{\frac{x}{2a}}\right)\,dx
\leq \pi\sqrt{a}+\int_{7}^{+\infty}e^{-\frac{1}{2}(\frac{x^2}{2\pi}-x)}\,dx
\]
for all $a\in(0,1)$. The sum begins with the term related to $n=8$ for two reasons. On the one hand, the function $g:x\mapsto e^{-\frac{x^2}{4a\pi}}\cosh\left(\frac{x}{2a}\right)$ is differentiable on $\R$ and
\[
Dg(x)=\frac{e^{-\frac{x^2}{4a\pi}}}{2a}\left(-\frac{x}{\pi}\cosh\left(\frac{x}{2a}\right)+\sinh\left(\frac{x}{2a}\right)\right)\leq0\quad\Leftrightarrow\quad x\geq\pi\tanh\left(\frac{x}{2a}\right),
\]
which implies that $g$ is decreasing on $[4,+\infty)$. On the other hand, the function $x\mapsto \frac{x^2}{2\pi}-x$ is positive on $[7,+\infty)$. Consequently, there exists $A\in(0,1)$ such that, for all $a\in(0,A)$, we have
\[
\frac{\pi}{2C}\left|\sum_{n\in\Z} \frac{1}{\sqrt{a}}e^{-\frac{(\pi+n)^2}{4a\pi}}-1\right|>a^{\eta}
\]
and then
\begin{equation}\label{Lusin32}
|\mathcal{W}_\psi R(a,1)|>Ca^{1+\eta}.
\end{equation}
\smallskip

In fact, the Lusin wavelet has only one vanishing moment since $\hat{\psi}(0)=0$ and $(D\hat{\psi})(0)\neq 0$, because the function $x\mapsto x\psi(x)$ is not integrable on $\R$. Inequality~\eqref{Lusin32} thus shows that the second vanishing moment is essential for the study of the H\"{o}lder continuity of $R$ when the exponent is (strictly) greater than $1$. We could otherwise find $D>0$ and $\delta>0$ such that
\[
|\mathcal{W}_\psi R(a,b)|\leq D\,a^{3/2}\left(1+\left(\frac{|b-1|}{a}\right)^{3/2}\right)
\]
for all $a\in(0,\delta)$ and $b\in(1-\delta,1+\delta)$ and then $|\mathcal{W}_\psi R(a,1)|\leq D\,a^{3/2}$ for all $a\in(0,\delta)$, which is in contradiction with~\eqref{Lusin32} by taking $C=D$, $\eta=1/2$ and $a\in(0,\min\{\delta,A\})$.

\section{Appendix}

Let us give a proof of Theorem~\ref{LusinRecons}. It is based on the ideas of~\cite{Da,HT,Ho} and adapted to the case of the continuous wavelet transform related to a wavelet~$\psi$ which belongs to~$H^2(\R)$.

\begin{proof}[Proof of Theorem~\ref{LusinRecons}]
Let us fix $x\in\R$ and $r>\varepsilon>0$. We write
\[
f_{\varepsilon,r}(x)=\int_\varepsilon^r\left(\int_{-\infty}^{+\infty}\mathcal{W}_\psi f(a,b)\,\frac{1}{a}\varphi\left(\frac{x-b}{a}\right)\,db\right)\frac{1}{a}\,da
\]
and we have
\[
f_{\varepsilon,r}(x)=(M_{\varepsilon,r}\star f)(x)
\]
by Fubini's theorem, where $M_{\varepsilon,r}$ is defined by
\[
M_{\varepsilon,r}(t)=\int_\varepsilon^r\left(\int_{-\infty}^{+\infty}\overline{\psi}\left(-\frac{b}{a}\right)\varphi\left(\frac{t-b}{a}\right)\,db\right)\frac{1}{a^3}\,da,\quad t\in\R.
\]

Since $M_{\varepsilon,r}\in L^1(\R)$ and the support of $\hat{\psi}$ is included in $(0,+\infty)$, we have
\[
\hat{M}_{\varepsilon,r}(\xi)
=\int_\varepsilon^r \overline{\hat{\psi}}(a\xi)\hat{\varphi}(a\xi)\frac{1}{a}\,da
=\left\{
\begin{array}{ll}
0&\text{if}\;\xi\leq 0\\
\displaystyle\int_{\varepsilon\xi}^{r\xi}\overline{\hat{\psi}}(a)\hat{\varphi}(a)\frac{1}{a}\,da&\text{if}\;\xi>0
\end{array}
\right..
\]
Moreover, we have
\begin{equation}\label{M}
\hat{M}_{\varepsilon,r}(\xi)
=m(\varepsilon\xi)-m(r\xi)
\end{equation}
for all $\xi\in\R$, where $m$ is defined by
\[
m(\xi)
=\left\{\begin{array}{ll}
\displaystyle\int_\xi^{+\infty}\overline{\hat{\psi}}(a)\hat{\varphi}(a)\frac{1}{a}\,da&\text{if}\;\xi\geq 0\\
\displaystyle\int_{-\xi}^{+\infty}\overline{\hat{\psi}}(-a)\hat{\varphi}(-a)\frac{1}{a}\,da&\text{if}\;\xi< 0
\end{array}\right..
\]
It is easy to check that $m(0)=1$, $m=0$ on $(-\infty,0)$ and that $m$ is continuous only on~$\R\setminus\{0\}$. Since we have the three following properties: $\hat{\psi}$ is bounded, $\varphi$ is differentiable and $D\varphi\in L^2(\R)$, we obtain
\[
|m(\xi)|
\leq \left(\int_0^{+\infty}|a\hat{\varphi}(a)|^2da\right)^{1/2}\left(\int_\xi^{+\infty}\frac{|\hat{\psi}(a)|^2}{a^4}\,da\right)^{1/2}
\leq \frac{C'}{\xi^{3/2}}
\]
for all $\xi>0$, by Cauchy-Schwarz inequality, where $C'$ is a positive constant. Then, $m$ is bounded and there exists $C>0$ such that
\[
|m(\xi)|\leq\frac{C}{(1+|\xi|)^{3/2}}
\]
for all $\xi\in\R$. So $m\in L^1(\R)\cap L^2(\R)$ and we can define $M$ by $M(\xi)=\hat{m}(-\xi)/\pi$ for all $\xi\in\R$. By definition, $M$ is continuous and bounded on $\R$.

Moreover, $m$ is differentiable on~$\R\setminus\{0\}$ and
\[
Dm(\xi)=\left\{
\begin{array}{ll}
0&\text{if}\;\xi<0\\
\displaystyle-\overline{\hat{\psi}}(\xi)\hat{\varphi}(\xi)\frac{1}{\xi}&\text{if}\;\xi>0
\end{array}
\right..
\]
Since $\hat{\varphi}(0)=0$ and $x\mapsto x\varphi(x)$ is integrable on~$\R$, we have
\[
\hat{\varphi}(\xi)
=\left|\int_{\R}\varphi(x)\left(e^{-ix\xi}-1\right)dx\right|
=\left|\int_{\R}x\varphi(x)\left(\int_0^\xi -ie^{-ixt}\,dt\right)dx\right|
\leq C''|\xi|
\]
for all $\xi\in\R$, where $C''$ is a positive constant. Consequently, $Dm\in L^2(\R)$ because $\psi\in L^2(\R)$. So $M\in L^1(\R)$ since we can write $M$ as the product of two square integrable functions: for all $x\in\R$, we have
\[
M(x)=\frac{1}{\sqrt{1+x^2}}\left(\sqrt{1+x^2}\,M(x)\right),
\]
where the second factor is square integrable, because $m$ and $Dm$ are square integrable on~$\R$. Moreover, by the Dirichlet condition for Fourier inversion theorem (since $m$ and $Dm$ are piecewise continuous), we have
\[
\int_{\R} M(x)\,dx=\hat{M}(0)=m(0^+)+m(0^-)=1
\]
using~\eqref{CondAdm} where $m(0^\pm)=\lim_{\xi\to 0^{\pm}}m(\xi)$.

By definition of $M$ and by Fourier inversion theorem in~\eqref{M}, we have
\[
M_{\varepsilon,r}(t)=\frac{1}{2}\left(\frac{1}{\varepsilon}M\left(\frac{t}{\varepsilon}\right)-\frac{1}{r}M\left(\frac{t}{r}\right)\right)
\]
for all $t\in\R$ and we then obtain
\[
f_{\varepsilon,r}(x)=\frac{1}{2}\left(\int_{\R}\frac{1}{\varepsilon}M\left(\frac{x-t}{\varepsilon}\right)f(t)\,dt-\int_{\R}\frac{1}{r}M\left(\frac{x-t}{r}\right)f(t)\,dt\right).
\]
The first integral converges to $f(x)$ as $\varepsilon$ tends to $0^+$ by Lebesgue theorem. The second integral converges to $0$ as $r$ tends to $+\infty$ thanks to Lemma 6.3.3 in~\cite{Ho}, because $f$ is bounded and weakly oscillating around on the origin, and $M\in L^1(\R)$ is of integral equal to $1$. The conclusion follows.
\end{proof}

\bigskip

\noindent L. Simons\\
Institute of Mathematics\\
University of Liège \\
Grande Traverse 12 (Bât. 37)\\
4000 Liège, Belgium\\
e-mail: \textit{L.Simons@ulg.ac.be}

\end{document}